\documentclass[reqno,12pt]{amsart}

\theoremstyle{plain}
\newtheorem{thm}{Theorem} [section]
\newtheorem{cor}[thm]{Corollary} 
\newtheorem{lemma}[thm]{Lemma} 
\newtheorem{prop}[thm]{Proposition}

\theoremstyle{remark}

\theoremstyle{definition}

\newtheorem{notation}[thm]{Notation}

\usepackage{amscd,amssymb,comment,epic,eepic,euscript,graphics,color}
\usepackage{enumerate}
\usepackage[initials]{amsrefs}

\newcount\theTime
\newcount\theHour
\newcount\theMinute
\newcount\theMinuteTens
\newcount\theScratch
\theTime=\number\time
\theHour=\theTime
\divide\theHour by 60
\theScratch=\theHour
\multiply\theScratch by 60
\theMinute=\theTime
\advance\theMinute by -\theScratch
\theMinuteTens=\theMinute
\divide\theMinuteTens by 10
\theScratch=\theMinuteTens
\multiply\theScratch by 10
\advance\theMinute by -\theScratch

\def\today{{\number\day\space
 \ifcase\month\or
  January\or February\or March\or April\or May\or June\or
  July\or August\or September\or October\or November\or December\fi
 \space\number\year}}

\newcommand\Cpx{{\mathbb C}}
\newcommand\diag{\text{\rm diag}}
\newcommand\eps{{\varepsilon}}
\newcommand\lambdat{{\tilde\lambda}}
\newcommand\Ibar{{\overline I}}

\newcommand\Kbar{{\overline K}}
\newcommand\Kt{{\widetilde K}}
\newcommand\lspan{\mathrm{span}\,}
\newcommand\Mcal{{\mathcal{M}}}
\newcommand\mut{{\tilde{\mu}}}
\newcommand\Mt{{\widetilde M}}
\newcommand\Nats{{\mathbb N}}
\newcommand\nut{{\tilde{\nu}}}
\newcommand\Reals{{\mathbb R}}
\newcommand\Proj{{\mathrm{Proj}}}
\newcommand\SVF{C^{(r)}_{\mathbb{R}}[0,1]_+^\downarrow} 
\newcommand\tdet{{\textstyle\det}}
\newcommand\Tr{{\mathrm{Tr}}}

\begin{document}

\title[Characterization of singular numbers]{Characterization of singular numbers of products of operators in matrix algebras and finite von Neumann algebras}

\author[Bercovici]{H.\ Bercovici$^1$}
\address{H.\ Bercovici, Department of Mathematics, Indiana University, Bloomington, IN 47405, USA}
 \thanks{\footnotesize $^1$Research supported in part by a grant from the NSF}
 \email{bercovic@indiana.edu}
\author[Collins]{B.\ Collins$^2$}
 \address{B.\ Collins, D\'epartement de Math\'ematique et Statistique, Universit\'e d'Ottawa,
585 King Edward, Ottawa, ON, K1N6N5 Canada,
WPI AIMR, Tohoku, Sendai, 980-8577 Japan
and 
CNRS, Institut Camille Jordan Universit\'e  Lyon 1,
France}
 \email{bcollins@uottawa.ca}
 \thanks{\footnotesize $^2$Research supported in part by ERA, NSERC and AIMR}
 \author[Dykema]{K.\ Dykema$^1$}
 \address{K.\ Dykema, Department of Mathematics, Texas A\&M University,
 College Station, TX 77843-3368, USA}
 \email{kdykema@math.tamu.edu}
\author[Li]{W.S.\ Li$^1$}
\address{W.S.\ Li, School of Mathematics, Georgia Institute of Technology, Atlanta, GA 30332-1060, USA}
  \email{li@math.gatech.edu}
 
 
\date{June 25, 2013}

\begin{abstract}
We characterize in terms
of inequalities the possible generalized singular numbers of a product $AB$ of operators $A$ and $B$
having given generalized singular numbers,
in an arbitrary finite von Neumann algebra.
We also solve the analogous problem in matrix algebras $M_n(\Cpx)$, which seems to be new insofar as we do not
require $A$ and $B$ to be invertible.
\end{abstract}

\maketitle

\section{Introduction and summary of results}

H.\ Weyl~\cite{W12} asked:
what are the possible eigenvalues of $A+B$ when $A$ and $B$ are $n\times n$ 
symmetric matrices whose eigenvalues are given?
A complete answer was conjectured by 
A.\ Horn~\cite{H62}.
This problem became known as the Horn problem, attracted the attention of many mathematicians, and was finally
solved (proving Horn's conjecture) with the critical input of 
A. A.\ Klyachko~\cite{Kl98}
and Knutson and Tao~\cite{KT99}.
We refer to \cite{Fulton00} for a description of the results. 
Later, these results have been extended in many directions. Let us mention two of them 
of interest for this paper: the multiplicative direction, and the infinite dimensional direction.

In the multiplicative direction, one problem is to describe the possible singular values of 
$AB$ when $A$ and $B$ are $n\times n$ 
symmetric matrices whose eigenvalues are given. 
We will call this the {\em multiplicative Horn problem}.
This problem is fully solved in the case where $A$ and $B$
are invertible matrices, because
Kyachko~\cite{Kl00} showed it is equivalent to the additive problem after taking logarithms.
We could not find a proof in the literature for the
case where $A$ and $B$ need not be invertible and we solve it here.
The solution to this problem is a limit of the invertible case,
but the description is perhaps
not completely obvious;
the proof relies on the solution of the invertible case and on an interpolation result from~\cite{BLT09}.

In the infinite dimensional direction
(meaning here,
in infinite dimensional von Neumann algebras that have normal, faithful traces; thus, in so-called finite von Neumann algebras),
it was proved in~\cite{BCDLT10} that the 
solution to the additive Horn problem essentially survives, after natural adaptation to
the infinite dimensional setting. 

The main result of this paper is to settle the multiplicative Horn problem in the setting of finite von Neumann algebras.
Similar to the additive case, the result is an
infinite dimensional modification of the finite dimensional case.

In von Neumann algebras that have the Connes embedding property (namely, those that embed in an ultrapower $R^\omega$
of the hyperfinite II$_1$-factors, that is, where finite tuples of elements can be approximated in mixed moments by complex matrices)
it is not so difficult to see, nor is it surprising, that the additive and multiplicative Horn problems have solutions described by the obvious
infinite dimensional modifications of the finite dimensional solutions (see~\cite{BL01}).
Thus, our main result shows that,
as in the additive case, the solution of the multiplicative Horn problem in arbitrary finite 
von Neumann algebras is equivalent to its counterpart in finite von Neumann algebras
that have the Connes embedding property.
It is not known whether all finite von Neumann algebras 
with separable predual have the Connes embedding property.

The main theorems of this paper are Theorems~\ref{thm:KKt} and \ref{thm:MMt}.
The paper is organized as follows:
Section~\ref{sec:prelimMats} contains preliminaries about Horn triples and matrix algebras;
Section~\ref{sec:singMat} proves inequalities involving singular numbers in matrix algebras;
Section~\ref{sec:non-inv} gives the solution of the multiplicative Horn problem for all (including non-invertible) matrices;
Section~\ref{sec:preliminaries} contains some preliminaries about finite
von Neumann algebras and Horn triples;
Section~\ref{sec:sing} proves inequalities involving singular numbers of products in finite von Neumann algebras;
Section~\ref{sec:finvN} gives the solution of the multiplicative Horn problem in finite von Neumann algebras.

\section{Preliminaries on Horn triples and matrix algebras}
\label{sec:prelimMats}

\subsection{Horn triples and additive Horn inequalities}
\label{subsec:HornTrips}

Given integers $1\le r\le n$, to each triple $(I,J,K)$ of subsets of $\{1,\ldots,n\}$, each of cardinality $r$,
and writing
\begin{align*}
I&=\{i(1)<i(2)<\cdots<i(r)\}, \\
J&=\{j(1)<j(2)<\cdots<j(r)\}, \\
K&=\{k(1)<k(2)<\cdots<k(r)\},
\end{align*}
if the identity
\begin{equation}\label{eq:IJKident}
\sum_{\ell=1}^r\big((i(\ell)-\ell)+(j(\ell)-\ell)+(k(\ell)-\ell)\big)=2r(n-r),
\end{equation}
holds,
then one associates the Littlewood--Richardson coefficient $c_{IJK}$, which is a nonnegative integer.
(See, e.g., Fulton~\cite{Fulton00} for more on this, though note that his $K$ in triples $(I,J,K)$ corresponds
to our $\overline K$, under the operation defined below by~\eqref{eq:Ibar}.)
Supposing $A,B$ and $C$ are $n\times n$ Hermitian matrices with eigenvalues (listed according to multiplicity
and in decreasing order) $\alpha=(\alpha_1,\ldots,\alpha_n)$ for $A$,
$\beta=(\beta_1,\ldots,\beta_n)$ for $B$ and $\gamma=(\gamma_1,\ldots,\gamma_n)$ for $C$,
the corresponding Horn inequality is
\begin{equation}\label{eq:Horn}
\sum_{i\in I}\alpha_i+\sum_{j\in J}\beta_j+\sum_{k\in K}\gamma_k\le0.
\end{equation}
Horn's conjecture is equivalent to the assertion
that the set of all triples of eigenvalue sequences $(\alpha,\beta,\gamma)$ arising from $n\times n$ Hermitian matrices
$A$, $B$ and $C$ subject to $A+B+C=0$
equals the set of triples $(\alpha,\beta,\gamma)$ such that
the equality
\begin{equation}\label{eq:Tr=}
\sum_{i=1}^n\alpha_i+\beta_i+\gamma_i=0
\end{equation}
holds and
the inequality~\eqref{eq:Horn} holds for all $(I,J,K)$ such that $c_{IJK}>0$.
Belkale~\cite{Be01} showed that the inequalities with $c_{IJK}>1$ are redundant, so that Horn's conjecture concerns
the convex body determined by the inequalities~\eqref{eq:Horn}
for all triples $(I,J,K)$ with $c_{IJK}=1$.

For future use, we let $H(n,r)$ be the set of all triples $(I,J,K)$ as described above, satisfying~\eqref{eq:IJKident} and with $c_{IJK}=1$.
For convenience, we also declare that $(\varnothing,\varnothing,\varnothing)$ is a Horn triple and let
$H(n,0)=\{(\varnothing,\varnothing,\varnothing)\}$.

Given $n\in\Nats$ and a subset $K$ of $\{1,\ldots,n\}$, following \cite{Fulton00} we let
\begin{equation}\label{eq:Ibar}
\Kbar=\{n+1-k\mid k\in K\}.
\end{equation}
Letting $D=-C$ and letting $\rho$ be the eigenvalue sequence of $D$,  we have $\gamma_k=\rho_{n+1-k}$ and
from~\eqref{eq:Horn} and~\eqref{eq:Tr=} we get the two inequalities
\begin{gather}
\sum_{i\in I}\alpha_i+\sum_{j\in J}\beta_j\le\sum_{k\in\Kbar}\rho_k \label{eq:AddHorn<} \\
\sum_{i\in I^c}\alpha_i+\sum_{j\in J^c}\beta_j\ge\sum_{k\in\Kbar^c}\rho_k, \label{eq:AddHorn>}
\end{gather}
where in~\eqref{eq:AddHorn>} the complements $I^c$, $J^c$ and $\Kbar^c$ are taken in $\{1,\ldots,n\}$.
These may be called the additive Horn inequalities for $D=A+B$.

\subsection{The intersection property in matrix algebras}
\label{subsec:IPmat}

A {\em full flag} in $\Cpx^n$ is an increasing sequence $E=\{E_1,\ldots,E_n\}$ of subspaces of $\Cpx^n$
such that $\dim(E_j)=j$ for all $j$.

Given $n\in\Nats$ and a set $I=\{i(1),\ldots,i(r)\}$ with $1\le i(1)<i(2)\cdots<i(r)\le n$,
and given a full flag $E$ in $M_n(\Cpx)$,
we consider the corresponding {\em Schubert variety} $S(E,I)$.
It is the set of all subspaces $V\subseteq\Cpx^n$ of dimension $r$ such that for all $\ell\in\{1,\ldots,r\}$,
\[
\dim(V\cap E_{i(\ell)})\ge\ell.
\]

Given  subsets $I$, $J$ and $K$ of $\{1,\ldots,n\}$, each of cardinality $r$, we say the triple $(I,J,K)$
has the {\em intersection property} in $M_n(\Cpx)$ if,
whenever $E$, $F$ and $G$ are full flags in $\Cpx^n$,
the intersection $S(E,I)\cap S(F,J)\cap S(G,K)$ is nonempty.

It is well known that every triple $(I,J,K)\in\bigcup_{r=0}^n H(n,r)$ has the intersection property.
In fact, a construction in~\cite{BCDLT10}
provides for every $(I,J,K)\in H(n,r)$ an algorithm to find the projection onto a subspace belonging to
$S(E,I)\cap S(F,J)\cap S(G,K)$
as a lattice polynomial of the projections onto the subspaces in the flags $E$, $F$ and $G$,
provided that the latter are in general position.

\section{Inequalities for singular numbers in matrix algebras}
\label{sec:singMat}

In this section, we prove inequalities involving singular numbers of a product of matrices
(Proposition~\ref{prop:ABCMn}).
This can be viewed as a template for the proof of an analogous inequality for singular numbers in finite
von Neumann algebras that is found in Section~\ref{sec:sing},
but the result is also used in Section~\ref{sec:non-inv} to help solve the multiplicative Horn problem for (non-invertible) matrices.

Recall that the singular numbers of an $n\times n$ complex matrix $A\in M_n(\Cpx)$
are $\|A\|=s_1(A)\ge s_2(A)\ge\cdots\ge s_n(A)\ge0$,
where
\begin{equation}\label{eq:defsingMn}
s_j(A)=\inf\|A(1-P)\|,
\end{equation}
with the infimum over all self-adjoint projections $P$ of rank $j-1$.
In other words, singular numbers of $A$ are the eigenvalues of $|A|=(A^*A)^{1/2}$, listed
according to multiplicity and in decreasing order.
Note also, for $A\in M_n(\Cpx)$, we have
\begin{equation}\label{eq:absdet}
|\tdet(A)|=\tdet(|A|)=\prod_{j=1}^ns_j(A)
\end{equation}
and if $A$ is invertible, then
\[
|\tdet(A)|=\exp\big(\Tr(\log|A|)\big),
\]
where $\Tr$ is the unnormalized trace.

\begin{lemma}\label{lem:detsnum}
Suppose $A\in M_n(\Cpx)$.
Let $v_1,\ldots,v_n$ be an orthonormal set of eigenvectors of $|A|$ with corresponding eigenvalues
$s_1(A),\ldots,s_n(A)$, respectively.
Consider the full flag $E_1\subsetneq\cdots\subsetneq E_n$ given by
\[
E_k=\lspan\{v_1,v_2,\ldots,v_k\}.
\]
Let $I=\{i(1),\ldots,i(r)\}\subseteq\{1,\ldots,n\}$, with $i(1)<i(2)<\cdots<i(r)$,
suppose $V\subseteq\Cpx^n$ is a subspace of dimension $r$ with
\[
\dim(V\cap E_{i(\ell)})\ge\ell\quad(1\le\ell\le r)
\]
and let $P\in M_n(\Cpx)$ be the self-adjoint projection onto $V$.
Let $Q$ be the self-adjoint projection onto an $r$-dimensional subspace containing $A(V)$.
Let $W\in M_n(\Cpx)$ be any partial isometry such that $WW^*=P$ and $W^*W=Q$.
Let $\tdet_P$ denote the determinant on the algebra $PM_n(\Cpx)P$, which is unitarily isomorphic to the $r\times r$ matrices.
Then
\[
\big|\tdet_P(WAP)\big|\ge s_{i(1)}(A)s_{i(2)}(A)\cdots s_{i(r)}(A).
\]
\end{lemma}
\begin{proof}
Since $|\tdet_P(WAP)|=\prod_{\ell=1}^rs_\ell(WAP)$,
it will suffice to show
\begin{equation}\label{eq:sineq}
s_\ell(WAP)\ge s_{i(\ell)}(A)\quad(\ell\in\{1,\ldots,r\}).
\end{equation}
Suppose $F\subseteq V$ is a subspace of dimension $\ell-1$.
Since its orthocomplement $V\ominus F$ has dimension $r-\ell+1$ and since the dimension of $V\cap E_{i(\ell)}$
is as least $\ell$, there is a unit vector $x$ in $(V\ominus F)\cap E_{i(\ell)}$.
Then $\|WAPx\|=\|Ax\|\ge s_{i(\ell)}$.
This proves~\eqref{eq:sineq}.
\end{proof}

\begin{notation}
We call a flag $E=(E_1,E_2,\ldots,E_n)$ as in Lemma~\ref{lem:detsnum} an
{\em eigenvector flag} of $|A|$.
\end{notation}

Though the following proposition follows from Klyachko's Theorem (given as Theorem~\ref{thm:klyachko} below),
we will give a proof here because it provides a model for the proof of the finite von Neumann algebra
version, Theorem~\ref{thm:ABC}. 
\begin{prop}\label{prop:ABCMn}
Let $A,B,C\in M_n(\Cpx)$ be such that $ABC=1_n$.
Suppose $(I,J,K)\in H(n,r)$ for some $0\le r\le n$.
Then
\begin{equation}\label{eq:ABCMn}
\prod_{\ell=1}^rs_{i(\ell)}(A)s_{j(\ell)}(B)s_{k(\ell)}(C)\le 1.
\end{equation}
\end{prop}
\begin{proof}
If $r=0$ then $(I,J,K)=(\varnothing,\varnothing,\varnothing)$ and~\eqref{eq:ABCMn} is by definition an equality.
So we may suppose $r\ge1$.
Let $E$, $F$ and $G$ be eigenvector flags of $|A|$, $|B|$ and $|C|$, respectively.
We choose a subspace $V\subseteq\Cpx^n$ of dimension $r$ such that
\[
\dim(BC(V)\cap E_{i(\ell)})\ge\ell,\quad\dim(C(V)\cap F_{j(\ell)})\ge\ell,\quad\dim(V\cap G_{k(\ell)})\ge\ell,
\quad(1\le\ell\le r).
\]
This can be done by applying the intersection property to the flags $C^{-1}B^{-1}(E)$, $C^{-1}(F)$ and $G$.
Let $P$, $Q$ and $R$ be the self-adjoint projections onto $V$, $C(V)$ and $BC(V)$, respectively.
Let $W_{Q,P}$ and $W_{R,P}$ be partial isometries such that $W_{Q,P}^*W_{Q,P}=P=W_{R,P}^*W_{R,P}$,
$W_{Q,P}W_{Q,P}^*=Q$ and $W_{R,P}W_{R,P}^*=R$.
Note that, since $ABC=1_n$, we have $AR=PAR$.
Then applying Lemma~\ref{lem:detsnum} three times, we have
\begin{align*}
1&=\tdet_P(ABCP)=\tdet_P(PARBQCP) \\[1ex]
&=\tdet_P(PAW_{R,P})\tdet_P(W_{R,P}^*BW_{Q,P})\tdet_P(W_{Q,P}^*CP) \\[1ex]
&=\tdet_P(W_{R,P}AR)\tdet_Q(W_{Q,P}W_{R,P}^*BQ)\tdet_P(W_{Q,P}^*CP) \\
&\ge\prod_{\ell=1}^rs_{i(\ell)}(A)s_{j(\ell)}(B)s_{k(\ell)}(C),
\end{align*}
as required.
\end{proof}

\begin{cor}\label{cor:MnInvIneqs}
If $A,B\in M_n(\Cpx)$ are invertible and if $D=AB$, then for every $(I,J,K)\in\bigcup_{r=0}^n H(n,r)$, we have
\begin{gather}
\sum_{i\in I}\log s_i(A)+\sum_{j\in J}\log s_j(B)\le\sum_{k\in\Kbar}\log s_k(D) \label{eq:Mnsineq<} \\
\sum_{k\in\Kbar^c}\log s_k(D)\le \sum_{i\in I^c}\log s_i(A)+\sum_{j\in J^c}\log s_j(B), \label{eq:Mnsineq>}
\end{gather}
where in~\eqref{eq:Mnsineq>}, $\Kbar^c$, $I^c$ and $J^c$ indicate the respective complements in $\{1,\ldots,n\}$.
\end{cor}
\begin{proof}
To obtain~\eqref{eq:Mnsineq<}, we apply Proposition~\ref{prop:ABCMn} with $C=D^{-1}$ and observe that we have
$s_k(C)=s_{n+1-k}(D)^{-1}$.
Now~\eqref{eq:Mnsineq>} follows from~\eqref{eq:Mnsineq<}, using~\eqref{eq:absdet}
and $\det(D)=\det(A)\det(B)$.
\end{proof}

\begin{thm}\label{thm:MnIneqs}
Let $A,B\in M_n(\Cpx)$ and let $D=AB$.
Then for every $(I,J,K)\in\bigcup_{r=0}^n H(n,r)$,
the inequalities~\eqref{eq:Mnsineq<} and~\eqref{eq:Mnsineq>} hold, with $-\infty$ allowed for values.
\end{thm}
\begin{proof}
Writing $A=U|A|$ and $B=|B^*|V$ for unitaries $U$ and $V$ we have $D=U|A||B^*|V$ and $s_i(A)=s_i(|A|)$, $s_j(B)=s_j(|B^*|)$
and $s_k(D)=s_k(|A||B^*|)$.
Thus, we may without loss of generality assume $A\ge0$ and $B\ge0$.
Let $\eps_1,\eps_2>0$ and let $D(\eps_1,\eps_2)=(A+\eps_11)(B+\eps_21)$.
Then $s_i(A+\eps_11)=\eps_1+s_i(A)$ and $s_j(B+\eps_21)=\eps_2+s_j(B)$
and from Corollary~\ref{cor:MnInvIneqs} we get
\begin{gather}
\sum_{i\in I}\log(\eps_1+s_i(A))+\sum_{j\in J}\log(\eps_2+s_j(B))\le\sum_{k\in\Kbar}\log s_k(D(\eps_1,\eps_2)) \label{eq:sumlog<Deps} \\
\sum_{k\in\Kbar^c}\log s_k(D(\eps_1,\eps_2))\le\sum_{i\in I^c}\log(\eps_1+s_i(A))+\sum_{j\in J^c}\log(\eps_2+s_j(B)). \label{eq:sumlogDeps<}
\end{gather}
But letting  $\eps_1,\eps_2\to0$, we have $s_k(D(\eps_1,\eps_2))\to s_k(D)$ for each $k$,
and from~\eqref{eq:sumlog<Deps} and~\eqref{eq:sumlogDeps<}
we obtain the desired inequalities~\eqref{eq:Mnsineq<} and~\eqref{eq:Mnsineq>}, with possible values $-\infty$.
\end{proof}

\section{The multiplicative Horn problem for non-invertible matrices}
\label{sec:non-inv}

We let $\Reals^{n\downarrow}$ and, respectively, $\mathbb{R}_+^{n\downarrow}$ and $\mathbb{R}_+^{*n\downarrow}$
denote the sets of nonincreasing sequences real numbers and, respectively,
of nonnegative real numbers and of strictly positive real numbers,
having length $n$.
We will need the following theorem, which follows from the solution of the additive Horn problem and a result
of Kyachko~\cite{Kl00} (see Theorem 2 of~\cite{Sp05}).
Again, $H(n,r)$ is the set of Horn triples with Littlewood--Richardson coefficient equal to $1$, as
described in~\S\ref{subsec:HornTrips}.

\begin{thm}\label{thm:klyachko}
For sequences
\[
\lambda = (\lambda_1, \ldots , \lambda_n),\quad
\mu = (\mu_1, \ldots , \mu_n),\quad
\gamma = (\gamma_1, \ldots , \gamma_n)
\]
in $\mathbb{R}_+^{*n\downarrow}$,
there exist matrices $A,B,C\in M_n(C)$
such that $ABC=1_n$ 
and having singular values of $\lambda,\mu$ and $\gamma$, respectively,
if and only if the following hold:
\begin{equation}\label{eq:detID}
\prod_{i=1}^n\lambda_i\mu_i\gamma_i=1
\end{equation}
and
\begin{equation}\label{eq:multHornIneq}
\forall(I,J,K)\in\bigcup_{r=0}^n H(n,r),\quad
\prod_{i\in I }\lambda_i\prod_{j\in J}\mu_j\prod_{k\in K}\gamma_k\leq 1.
\end{equation}
\end{thm}

This theorem fully solves the multiplicative Horn problem in $M_n(\Cpx)$ in the case
where $A,B,C$ are invertible.
In particular, the multiplicative Horn problem for invertibles is equivalent to the additive Horn
problem by taking logarithms.

In this section we will deal with the non-invertible case, namely:
given $\lambda,\mu\in\mathbb{R}_+^{n\downarrow}$,
what are the possible singular values $\nu\in\mathbb{R}_+^{n\downarrow}$ of $AB$
when $A$ and $B$ in $M_n(\Cpx)$
have respective singular values $\lambda$ and $\mu$?
We will denote the set of all such $\nu$ by $K_{\lambda,\mu}$ and call it the {\em multiplicative Horn body}.
It is equal the set of all singular values of matrices $\diag(\lambda)U\diag(\mu)$,
where $U$ ranges over the $n\times n$ unitary group.
To summarize, our goal in this section is to describe the set
\[
K_{\lambda,\mu}:=\{\nu\in\mathbb{R}_+^{n\downarrow}\mid \nu = \text{ singular values of }
\diag(\lambda)U\diag(\mu),\,U\in\mathbb{U}_n\}.
\]
As one might expect, the answer is a continuous limit of the invertible case, 
however with one subtlety in its description.
Note that, since the singular values of a matrix are continuous with respect to operator norm and since the unitary group is compact,
$K_{\lambda,\mu}$ is a compact subset of $\mathbb{R}_+^{n\downarrow}$.

See~\eqref{eq:Ibar} for the operation $I\mapsto\Ibar$.
Given $\lambda,\mu\in \mathbb{R}_+^{n\downarrow}$,
we let
\begin{align}
\Kt_{\lambda,\mu }=\bigg\{\nu\in \mathbb{R}_+^{n\downarrow}\;\bigg|\;&\forall(I,J,K)\in\bigcup_{r=0}^n H(n,r), \label{eq:Ktdef} \\
&\prod_{i\in I}\lambda_i\prod_{j\in J}\mu_j\leq \prod_{k\in \overline K}\nu_k \label{eq:Ktdef<} \\
&\text{ and }
\prod_{i\in I^c}\lambda_i\prod_{j\in J^c}\mu_j\geq \prod_{k\in\Kbar^c}\nu_k\bigg\}. \label{eq:Ktdef>}
\end{align}

\begin{lemma}\label{lem:invertibleCase}
If $\lambda,\mu \in \mathbb{R}_+^{*n\downarrow}$,
then $K_{\lambda,\mu}=\Kt_{\lambda,\mu}$.
\end{lemma}
\begin{proof}
This follows easily from Theorem~\ref{thm:klyachko}.
Indeed, $K_{\lambda,\nu}$ is the set of all $\nu$ arising from sequences $\gamma$ satisfying the
conditions~\eqref{eq:detID} and~\eqref{eq:multHornIneq} of Theorem~\ref{thm:klyachko},
where the correspondence between $\nu$ and $\gamma$ is given by
$\nu_i=\gamma_{n+1-i}^{-1}$.
Thus,
$K_{\lambda,\mu}$ is the set of all $\nu\in\mathbb{R}_+^{n\downarrow}$ such that
\begin{equation}\label{eq:nudetID}
\prod_1^n\nu_i=\prod_1^n\lambda_i\mu_i
\end{equation}
and
\begin{equation}\label{eq:numultHornIneq}
\forall(I,J,K)\in\bigcup_{r=0}^n H(n,r),\quad
\prod_{i\in I}\lambda_i\prod_{j\in J}\mu_j\leq \prod_{k\in \overline K}\nu_k
\end{equation}
If $\nu\in\Kt_{\lambda,\mu}$, then~\eqref{eq:numultHornIneq} holds by definition.
To see that~\eqref{eq:nudetID} holds, we apply the inequality~\eqref{eq:Ktdef<} 
for the Horn triple $(\{1,\ldots ,n\},\{1,\ldots ,n\},\{1,\ldots ,n\}   )\in H(n,n)$
and the inequality~\eqref{eq:Ktdef>} for the Horn triple $(\varnothing,\varnothing,\varnothing)\in H(n,0)$.
This yields $\Kt_{\lambda,\mu}\subseteq K_{\lambda,\mu}$.

For the reverse inclusion, assume $\nu\in K_{\lambda,\nu}$.
Then~\eqref{eq:Ktdef<} is~\eqref{eq:numultHornIneq}, while~\eqref{eq:Ktdef>}
follows from the~\eqref{eq:Ktdef<} and the determinant identity~\eqref{eq:nudetID}.
Thus, $\nu\in\Kt_{\lambda,\nu}$.
\end{proof}

We are now able to prove the main result of this section:
\begin{thm}\label{thm:KKt}
For all $\lambda,\mu \in\mathbb{R}_+^{n\downarrow}$, we have
$K_{\lambda,\mu}=\Kt_{\lambda,\mu}$.
\end{thm}
\begin{proof}
The inclusion $K_{\lambda,\mu}\subseteq\Kt_{\lambda,\mu}$ follows from Theorem~\ref{thm:MnIneqs}
by exponentiating.

It will be convenient to have the notation, applicable for any $\lambda,\mu\in\mathbb{R}_+^{n\downarrow}$, that 
$\Kt_{\lambda,\mu}^+$ is the set of all $\nu\in\Reals_+^{n\downarrow}$ so that~\eqref{eq:Ktdef<} holds for all Horn triples,
and
$\Kt_{\lambda,\mu}^-$ is the set of all $\nu\in\Reals_+^{n\downarrow}$ so that~\eqref{eq:Ktdef>} holds for all Horn triples.
Thus,
\[
\Kt_{\lambda,\mu}=\Kt_{\lambda,\mu}^-\cap\Kt_{\lambda,\mu}^+.
\]
For $\eps\in\Reals_+$ and $\nu\in\Reals_+^{n\downarrow}$, let $\nu+\eps\in\Reals_+^{n\downarrow}$
be obtained from $\nu$ by adding $\eps$ to all $n$ coordinates.
Similarly, if $\nu\in\Reals_+^{*n\downarrow}$, then $\log\nu\in\Reals^{n\downarrow}$ is obtained by taking the
logarithm of each component.

Now, to show $\Kt_{\lambda,\mu}\subseteq K_{\lambda,\mu}$, let $\nu\in\Kt_{\lambda,\mu}$ and let $\eps>0$.
Then $\nu+\eps\in\Kt_{\lambda,\mu}^+$ and for all $(I,J,K)\in\bigcup_{r=0}^n H(n,r)$, the inequality
\[
\prod_{i\in I}\lambda_i\prod_{j\in J}\mu_j\leq \prod_{k\in \overline K}(\nu+\eps)_k
\]
holds with strict inequality, except when $(I,J,K)=(\varnothing,\varnothing,\varnothing)$, when it is the equality $1=1$.
Similarly, we have $\nu\in\Kt_{\lambda+\eps,\mu+\eps}^-$ and all of the inequalities
\[
\prod_{i\in I^c}(\lambda+\eps)_i\prod_{j\in J^c}(\mu+\eps)_j\geq \prod_{k\in\Kbar^c}\nu_k
\]
hold with strict inequality, except when $(I^c,J^c,K^c)=(\varnothing,\varnothing,\varnothing)$, when it is the equality $1=1$.
Therefore, there is $\delta=\delta(\eps)$ satisfying $0<\delta<\eps$ such that
\[
\nu+\eps\in\Kt_{\lambda+\delta,\mu+\delta}^+\quad\text{ and }\quad
\nu+\delta\in\Kt_{\lambda+\eps,\mu+\eps}^-.
\]
Now $(\log(\lambda+\delta),\log(\mu+\delta),\log(\nu+\eps))$ satisfies the additive Horn inequalities~\eqref{eq:AddHorn<}
while
$(\log(\lambda+\eps),\log(\mu+\eps),\log(\nu+\delta))$ satisfies the additive Horn inequalities~\eqref{eq:AddHorn>}.
By the interpolation result, Proposition~3.2 of~\cite{BLT09}, it follows that there is $(\alpha,\beta,\rho)\in(\Reals^{n\downarrow})^3$
satisfying all of the inequalities for the additive Horn problem and such that componentwise we have
\begin{equation}\label{eq:abcineqs}
\begin{gathered}
\log(\lambda+\delta)\le\alpha\le\log(\lambda+\eps), \\
\log(\mu+\delta)\le\beta\le\log(\mu+\eps), \\
\log(\nu+\delta)\le\rho\le\log(\nu+\eps).
\end{gathered}
\end{equation}
Thus, letting
\[
\lambdat_\eps=\exp(\alpha),\quad\mut_\eps=\exp(\beta),\quad\nut_\eps=\exp(\rho),
\]
we have
$\nut_\eps\in\Kt_{\lambdat_\eps,\mut_\eps}$ and, by Lemma~\ref{lem:invertibleCase}, $\nut_\eps\in K_{\lambdat_\eps,\mut_\eps}$.
So there is
an $n\times n$ unitary matrix $U_\eps$ so that the singular numbers of $\diag(\lambdat_\eps)U_\eps\diag(\mut_\eps)$ are
precisely $\nut_\eps$.
Of course, the inequalities~\eqref{eq:abcineqs} give, componentwise,
\[
\lambda+\delta(\eps)\le\lambdat_\eps\le\lambda+\eps,\quad
\mu+\delta(\eps)\le\mut_\eps\le\mu+\eps,\quad
\nu+\delta(\eps)\le\nut_\eps\le\nu+\eps.
\]
Since the singular numbers of an $n\times n$ matrix are continuous with respect to the operator norm,  choosing by compactness
a sequence $\eps(k)$ tending to zero so that $U_{\eps(k)}$ converges as $k\to\infty$ in norm to a unitary matrix $U$
and taking the limit as $k\to\infty$
we obtain that
the singular numbers of $\diag(\lambda)U\diag(\mu)$ are precisely $\nu$.
Thus, $\nu\in K_{\lambda,\mu}$, as required.
\end{proof}

\section{Preliminaries in finite von Neumann algebras}
\label{sec:preliminaries}

Throughout this section and the next, $\Mcal$ will denote a finite von Neumann algebra with a normal, faithful, tracial state $\tau$.
We will also assume $\Mcal$ is diffuse, meaning that it has no minimal projections.
Now we recall some facts about finite von Neumann algebras and introduce notation that is used in the remainder of the paper.
Some of this notation (for example, related to flags and singular number) is in conflict with the notation used
for matrix algebras in previous sections.

\subsection{Projections and flags}
We let $\Proj(\Mcal)$ denote the set of (self-adjoint) projections in $\Mcal$ and for $P\ne0$ such a projection, the cut-down
von Neumann algebra $P\Mcal P$ will usually be taken with the tracial state $\tau(\cdot)/\tau(P)$.

Recall that for $P,Q\in\Proj(\Mcal)$, their greatest lower bound $P\wedge Q\in\Proj(\Mcal)$
has trace satisfying
\[
\tau(P\wedge Q)\ge\tau(P)+\tau(Q)-1.
\]
By relativising, we obtain the following easy but useful result.
\begin{lemma}\label{lem:EFP}
Let $E,F,P\in\Proj(\Mcal)$ with $E\le F$.
Then
\begin{equation}\label{eq:EFP}
\tau(P\wedge E)\ge\tau(P\wedge F)-\tau(F-E).
\end{equation}
\end{lemma}
\begin{proof}
Working in $F\Mcal F$, since $P\wedge E =(P\wedge F)\wedge E$, we have
\[
\tau^{(F\Mcal F)}(P\wedge E)\ge\tau^{(F\Mcal F)}(P\wedge F)+\tau^{(F\Mcal F)}(E)-1
\]
and multiplying by $\tau(F)$ yields~\eqref{eq:EFP}.
\end{proof}

A {\em flag} $E$ in $\Mcal$ will be a function $E:D\to\Proj(\Mcal)$ for some subset $D\subseteq[0,1]$
so that
for all $s,t\in D$ with $s\le t$, we have
$\tau(E(t))=t$ and $E(s)\le E(t)$.
A {\em full flag} is a flag whose domain $D$ is $[0,1]$.

If $T\in\Mcal$ and
$P\in\Proj(\Mcal)$, then we let $T\cdot P$ denote the range projection of $TP$.
The following properties are well known and easy to prove
(see, for example, \S2.2 of \cite{CD09} for this and more).
\begin{prop}
Let $S,T\in\Mcal$ 
and let $P,Q\in\Proj(\Mcal)$.
Then
\begin{enumerate}[(i)]
\item $S\cdot(T\cdot P)=(ST)\cdot P$;
\item in general, $\tau(T\cdot P)\le\tau(P)$,
while if $T$ is invertible or, more generally, if $T$ has zero kernel, then $\tau(T\cdot P)=\tau(P)$;
\item $T\cdot(P\wedge Q)=(T\cdot P)\wedge(T\cdot Q)$.
\end{enumerate}
\end{prop}

\subsection{Singular numbers and eigenvalue functions}
The singular numbers in the setting of a finite von Neumann algebra were introduced by von Neumann.
See Fack and Kosaki~\cite{FaKo86} for an excellent presentation and development of this theory.
For an element $A\in\Mcal$, the singular number function $s_A:[0,1]\to[0,\infty)$ is the right-continuous, nonincreasing function
defined by
\begin{equation}\label{eq:snums}
s_A(t)=\inf\{\|A(1-Q)\|\mid Q\in\Proj(\Mcal),\,\tau(Q)\le t\}.
\end{equation}
We may also write $s_A^{(\Mcal)}(t)$ if we want to indicate the von Neumann algebra.
Thus, for example, if $P\in\Proj(\Mcal)$ and $A\in P\Mcal P$, then we have
\begin{equation}\label{eq:sPMP}
s_A^{(\Mcal)}(t)=\begin{cases}
s_A^{(P\Mcal P)}(t/\tau(P)),&0\le t<\tau(P), \\
0,&\tau(P)\le t\le1.
\end{cases}
\end{equation}

Note that we have $s_A=s_{A^*}=s_{|A|}$ where $|A|=(A^*A)^{1/2}$.

For a self-adjoint $T\in\Mcal$, its {\em spectral distribution} $\mu_T$ is the Borel probability measure supported
on the spectrum of $T$ whose moments agree with those of $T$ (with respect to $\tau$).
It is also equal to $\tau$ composed with the projection valued spectral measure of $T$.

The {\em eigenvalue function} of $T$ 
is the nonincreasing, right-continuous function
$\lambda_T:[0,1)\to\Reals$ given by 
\[
\lambda_T(t)=\sup\{x\in\Reals\mid\mu_T((x,\infty))>t\}.
\]
There is a full flag $E_T$ of projections in $\Mcal$, that can be obtained by starting with a chain
of spectral projections of $T$ and extending in the case that the distribution $\mu_T$ has atoms,
so that
\[
T=\int_0^1\lambda_T(t)\,dE_T(t).
\]
We will call $E_T$ a {\em spectral flag} of $T$,
and the possible nonuniqueness of spectral flags will not concern us.
We note that, for $T\ge0$  $\lambda_T=s_T$ and all $t\in[0,1]$, we have
\begin{gather}
\|T(1-E_T(t))\|=s_T(t) \label{eq:Esnum} \\[1ex]
TE_T(t)\ge s_T(t)E_T(t). \label{eq:sEunder}
\end{gather}

\subsection{Fuglede--Kadison determinant}

The Fuglede--Kadison determinant~\cite{FuKa52} is the function $\Delta:\Mcal\to[0,\infty)$
defined by $\Delta(T)=\Delta(|T|)=\exp\tau(\log|T|)$ for $T$ invertible and
$\Delta(T)=\lim_{\eps\to0^+}\Delta(|T|+\eps)$ for $T$ non-invertible,
and it satisfies, for all $S,T\in\Mcal$,
$\Delta(ST)=\Delta(S)\Delta(T)$.
We may also write $\Delta^{(\Mcal)}(T)$ for $\Delta(T)$, to emphasize the von Neumann algebra (and, implicitly, the trace)
with respect to which the Fuglede--Kadison determinant is taken.

\subsection{The intersection property in II$_1$-factors}
\label{subsec:IP}

Given $n\in\Nats$ and a set $I=\{i(1),\ldots,i(r)\}$ with $1\le i(1)<i(2)\cdots<i(r)\le n$,
and given a flag $E$ in $\Mcal$ whose domain includes the rational numbers
\begin{equation}\label{eq:ratsn}
\left\{0,\frac1n,\frac2n,\ldots,\frac{n-1}n,1\right\},
\end{equation}
we consider the corresponding {\em Schubert variety} $S(E,I)$.
It is the set of all projections $P\in\Mcal$ satisfying $\tau(P)=r/n$ and, for all $\ell\in\{1,\ldots,r\}$,
\[
\tau\left(P\wedge E\left(\frac{i(\ell)}n\right)\right)\ge\frac\ell n.
\]

Given  subsets $I$, $J$ and $K$ of $\{1,\ldots,n\}$, each of cardinality $r$, we say the triple $(I,J,K)$
has the {\em intersection property} in $\Mcal$ if,
whenever $E$, $F$ and $G$ are flags in $\Mcal$ each of whose domains includes
the rational numbers~\eqref{eq:ratsn},
there is a projection $P\in S(E,I)\cap S(F,J)\cap S(G,K)$.
A main result of~\cite{BCDLT10} is that every $(I,J,K)\in H(n,r)$
has the intersection property in every II$_1$-factor $\Mcal$.

\section{Singular numbers of products in finite von Neumann algebras}
\label{sec:sing}

In this section, we prove inequalities for singular numbers of products in finite von Neumann algebras
that generalize those proved for matrix algebras in Section~\ref{sec:singMat}
(and whose proofs are also analogous).
Note that our results and techniques overlap with and extend those of Harada~\cite{H07}.

For the next two lemmas, we fix $n\in\Nats$ and $r\in\{1,\ldots,n\}$ and $I=\{i(1),\ldots,i(r)\}\subseteq\{1,\ldots,n\}$
such that $i(j)<i(j+1)$.
Consider the corresponding union of subintervals of $[0,1]$
\begin{equation}\label{eq:FI}
F_I=\bigcup_{\ell=1}^r\left[\frac{i(\ell)-1}n,\frac{i(\ell)}n\right]\subseteq[0,1]
\end{equation}
and let $m$ denote Lebesgue measure on $[0,1]$.
\begin{lemma}\label{lem:PEx}
Suppose $P\in\Proj(\Mcal)$ and $E$ is a full flag in $\Mcal$ and
\begin{equation}\label{eq:PFil}
\forall\ell\in\{1,\ldots,r\},\quad\tau\left(P\wedge E\left(\frac{i(\ell)}n\right)\right)\ge\frac\ell n.
\end{equation}
Then
\begin{equation}\label{eq:tauPEx}
\forall x\in[0,1],\quad\tau(P\wedge E(x))\ge m([0,x]\cap F_I).
\end{equation}
\end{lemma}
\begin{proof}
Since the left hand side is increasing in $x$ and the right hand side is constant when $x$ varies
over intervals disjoint from $F_I$, we need only prove~\eqref{eq:tauPEx} for $x\in\big[\frac{i(\ell)-1}n,\frac{i(\ell)}n\big]$,
$\ell\in\{1,\ldots,r\}$.
For such $x$, using Lemma~\ref{lem:EFP} and the hypothesis~\eqref{eq:PFil}, we get
\begin{align*}
\tau(P\wedge E(x))&\ge\tau\left(P\wedge E\left(\frac{i(\ell)}n\right)\right)
  -\tau\left(E\left(\frac{i(\ell)}n\right)-E(x)\right)\ge\frac\ell n-\frac{i(\ell)-x}n \\[1ex]
&=m\left(\left[0,\frac{i(\ell)}n\right]\cap F_I\right)-m\left(\left[x,\frac{i(\ell)}n\right]\cap F_I\right)=m([0,x]\cap F_I).
\end{align*}
\end{proof}

\begin{lemma}\label{lem:FKPAP}
Let $A\in\Mcal$, and let $E=E_{|A|}$ be a spectral flag of $|A|=(A^*A)^{1/2}$.
Let $P\in\Proj(\Mcal)$ be such that $\tau(P)=r/n$ and
\[
\forall\ell\in\{1,\ldots,r\},\quad\tau\left(P\wedge E\left(\frac{i(\ell)}n\right)\right)\ge\frac\ell n.
\]
Let $Q=A\cdot P$ and let $W\in\Mcal$ be any partial isometry such that $W^*W\ge Q$ and $WW^*=P$.
Then
\begin{equation}\label{eq:FKPAP}
\log\Delta^{(P\Mcal P)}(WAP)\ge\frac1{\tau(P)}\int_{F_I}\log(s_A(t))\,dt.
\end{equation}
\end{lemma}
\begin{proof}
By re-indexing the integrand and using~\eqref{eq:sPMP}, we get
\begin{align*}
\log\Delta^{(P\Mcal P)}(WAP)&=\int_{[0,1]}\log s^{(P\Mcal P)}_{WAP}(x)\,dx \\
&=\frac1{m(F_I)}\int_{F_I}\log s^{(P\Mcal P)}_{WAP}\left(\frac{m([0,t]\cap F_I)}{m(F_I)}\right)\,dt
\end{align*}
We will now prove that
\begin{equation}\label{eq:sWAP}
s^{(P\Mcal P)}_{WAP}\left(\frac{m([0,t]\cap F_I)}{m(F_I)}\right)\ge s_A^{(\Mcal)}(t)
\end{equation}
holds for almost all $t\in F_I$,
which will yield the desired inequality~\eqref{eq:FKPAP}.
We take $t\in F_I\backslash\partial F_I$ and we will show that for all $\eps>0$, we have
\begin{equation}\label{eq:sWAPeps}
s^{(P\Mcal P)}_{WAP}\left(\frac{m([0,t]\cap F_I)}{m(F_I)}\right)\ge s_A^{(\Mcal)}(t+\eps),
\end{equation}
which by right-continuity of $s_A$ will imply~\eqref{eq:sWAP}.
Suppose $Q\le P$ is a projection and $\tau(Q)\le m([0,t]\cap F_I)$.
By Lemma~\ref{lem:PEx}, we have 
\[
\tau(P\wedge E(t+\eps))\ge m([0,t+\eps]\cap F_I),
\]
and so, by Lemma~\ref{lem:EFP}, we have
\begin{multline*}
\tau(E(t+\eps)\wedge(P-Q))\ge\tau(E(t+\eps)\wedge P)-\tau(Q) \\
\ge  m([0,t+\eps]\cap F_I)-m([0,t]\cap F_I)=m([t,t+\eps]\cap F_I)>0.
\end{multline*}
By applying the operator $WA(P-Q)$ to a vector belonging to the range of
the projection $E(t+\eps)\wedge(P-Q))$ and using~\eqref{eq:sEunder}, we obtain $\|WA(P-Q)\|\ge s_A(t+\eps)$.
Hence, we have
\begin{multline*}
s^{(P\Mcal P)}_{WAP}\left(\frac{m([0,t]\cap F_I)}{m(F_I)}\right) \\[1ex]
=\inf\{\|WA(P-Q)\|\mid Q\in\Proj(P\Mcal P),\,\tau(Q)\le m([0,t]\cap F_I)\} \\
\ge s_A(t+\eps)
\end{multline*}
and~\eqref{eq:sWAPeps} is proved.
\end{proof}

The next results apply to Horn triples $(I,J,K)\in H(n,r)$ as described in~\S\ref{subsec:HornTrips},
using that the triple has the intersection property in every II$_1$-factor, as described in~\S\ref{subsec:IP}.

\begin{thm}\label{thm:ABC}
Let $A,B,C\in\Mcal$ be such that $ABC=1$.
For $r,n\in\Nats$ with $r\le n$, and for $(I,J,K)\in H(n,r)$,
we have
\begin{equation}\label{eq:intlogs}
\int_{F_I}\log s_A
+\int_{F_J}\log s_B
+\int_{F_K}\log s_C\le0,
\end{equation}
where $F_I$, $F_J$ and $F_K$ are the corresponding unions of subintervals of $[0,1]$ as defined in~\eqref{eq:FI}
and where the integrals are with respect to Lebesgue measure.
\end{thm}
\begin{proof}
Since every finite von Neumann algebra with specified normal, faithful, tracial state can be embedded into a II$_1$-factor
in a trace-preserving way (see, e.g., Prop.~8.1 of~\cite{BCDLT10}), we may without loss of generality assume that $\Mcal$ is a II$_1$-factor.

Consider the full flags
\[
E=(BC)^{-1}\cdot E_{|A|},\quad F=C^{-1}\cdot E_{|B|},\quad G=E_{|C|}.
\]
Since $(I,J,K)$ has the intersection property in $\Mcal$, there is a projection $P\in\Mcal$ with $\tau(P)=r/n$ such that
for all $\ell\in\{1,\ldots,r\}$ we have
\[
\tau\left(P\wedge E\left(\frac{i(\ell)}n\right)\right)\ge\frac\ell r,\quad
\tau\left(P\wedge F\left(\frac{j(\ell)}n\right)\right)\ge\frac\ell r,\quad
\tau\left(P\wedge G\left(\frac{k(\ell)}n\right)\right)\ge\frac\ell r.
\]
Let $Q=C\cdot P$ and $R=BC\cdot P$.
Since $C$ and $B$ are invertible, these projections have the same trace as $P$ and we can
choose partial isometries $W_{Q,P},W_{R,P}\in\Mcal$ such that
\[
W_{Q,P}^*W_{Q,P}=P,\quad W_{Q,P}W_{Q,P}^*=Q,\qquad W_{R,P}^*W_{R,P}=P,\quad W_{R,P}W_{R,P}^*=R.
\]
Then we have
\begin{align*}
1&=\Delta^{(P\Mcal P)}(PABCP)
=\Delta^{(P\Mcal P)}(PARBQCP) \\
&=\Delta^{(P\Mcal P)}(PAW_{R,P}W_{R,P}^*BW_{Q,P}W_{Q,P}^*CP) \\
&=\Delta^{(P\Mcal P)}(PAW_{R,P})\,\Delta^{(P\Mcal P)}(W_{R,P}^*BW_{Q,P})\,\Delta^{(P\Mcal P)}(W_{Q,P}^*CP).
\end{align*}
So
\begin{multline}\label{eq:ABClogsum}
0=\log\Delta^{(R\Mcal R)}(W_{R,P}AP)+\log\Delta^{(Q\Mcal Q)}(W_{Q,P}W_{R,P}^*BQ) \\
+\log\Delta^{(P\Mcal P)}(W_{Q,P}^*CP).
\end{multline}
But for all $\ell\in\{1,\ldots,r\}$ we have
\[
R\wedge E_{|A|}\left(\frac{i(\ell)}n\right)=(BC\cdot P)\wedge\left(BC\cdot E\left(\frac{i(\ell)}n\right)\right)
 =BC\cdot\left(P\wedge E\left(\frac{i(\ell)}n\right)\right),
\]
so
\[
\tau\left(R\wedge E_{|A|}\left(\frac{i(\ell)}n\right)\right)=\tau\left(P\wedge E\left(\frac{i(\ell)}n\right)\right)\ge\frac\ell r
\]
and, similarly,
\[
\tau\left(Q\wedge E_{|B|}\left(\frac{j(\ell)}n\right)\right)=\tau\left(P\wedge F\left(\frac{j(\ell)}n\right)\right)\ge\frac\ell r
\]
and, since $E_{|C|}=G$ we have
\[
\tau\left(P\wedge E_{|C|}\left(\frac{k(\ell)}n\right)\right)\ge\frac\ell r.
\]
Using~\eqref{eq:ABClogsum} and applying Lemma~\ref{lem:FKPAP} three times yields the desired inequality~\eqref{eq:intlogs}.
\end{proof}

See~\eqref{eq:Ibar} for the definition of the operation $K\mapsto\Kbar$.

\begin{cor}\label{cor:ABinv}
Let $A,B\in\Mcal$ be invertible and let $D=AB$.
With $(I,J,K)$ as in Theorem~\ref{thm:ABC}, we have
\begin{gather}
\int_{F_I}\log s_A+\int_{F_J}\log s_B\le\int_{F_\Kbar}\log s_D \label{eq:intlog<D} \\
\int_{(F_\Kbar)^c}\log s_D\le\int_{(F_I)^c}\log s_A+\int_{(F_J)^c}\log s_B, \label{eq:intlogD<}
\end{gather}
where the complements in~\eqref{eq:intlogD<} are taken in $[0,1]$.
\end{cor}
\begin{proof}
Applying Theorem~\ref{thm:ABC} with $C=D^{-1}$ and using that the equality
\[
\log s_D(t)=-\log s_C(1-t)
\]
holds for almost every $t\in[0,1]$,
(namely, at points of continuity) yields~\eqref{eq:intlog<D}.
Now using $\log\Delta(T)=\int_{[0,1]}s_T$ and $\Delta(D)=\Delta(A)\Delta(B)$, we get~\eqref{eq:intlogD<} from~\eqref{eq:intlog<D}.
\end{proof}

\begin{thm}\label{thm:AB}
Let $A,B\in\Mcal$ and let $D=AB$.
Then for all $r,n\in\Nats$ and all triples $(I,J,K)$ as in Theorem~\ref{thm:ABC},
the inequalities~\eqref{eq:intlog<D} and~\eqref{eq:intlogD<} hold, with $-\infty$ allowed for values.
\end{thm}
\begin{proof}
Writing $A=U|A|$ and $B=|B^*|V$ for unitaries $U$ and $V$ we have $D=U|A||B^*|V*$ and $s_A=s_{|A|}$, $s_B=s_{|B^*|}$
and $s_{D}=s_{|A||B^*|}$.
Thus, we may without loss of generality assume $A\ge0$ and $B\ge0$.
Let $\eps_1,\eps_2>0$ and let $D(\eps_1,\eps_2)=(A+\eps_11)(B+\eps_21)$.
Then $s_{A+\eps_11}=\eps_1+s_A$ and $s_{B+\eps_21}=\eps_2+s_B$
and from Corollary~\ref{cor:ABinv}, we get
\begin{gather}
\int_{F_I}\log(\eps_1+s_A)+\int_{F_J}\log(\eps_2+s_B)\le\int_{F_\Kbar}\log s_{D(\eps_1,\eps_2)} \label{eq:intlog<Deps} \\
\int_{(F_\Kbar)^c}\log s_{D(\eps_1,\eps_2)}\le\int_{(F_I)^c}\log(\eps_1+s_A)+\int_{(F_J)^c}\log(\eps_2+s_B). \label{eq:intlogDeps<}
\end{gather}

Since for any projection $P\in\Mcal$, and for $0<\eps_1'\le\eps_1$ we have
\[
\|D(\eps_1,\eps_2)(1-P)\|^2=\|(1-P)(B+\eps_2)(A^2+2\eps_1A+\eps_1^2)(B+\eps_2)(1-P)\|
\]
and
\[
A^2+2\eps_1'A+(\eps_1')^2\le A^2+2\eps_1A+\eps_1^2,
\]
we get $\|D(\eps'_1,\eps_2)(1-P)\|\le\|D(\eps_1,\eps_2)(1-P)\|$ and,
from the definition~\eqref{eq:snums} of singular numbers, we obtain that $s_{D(\eps_1,\eps_2)}$ is decreasing
in $\eps_1$.
However, since $D(\eps_1,\eps_2)$ and $(B+\eps_21)(A+\eps_11)$ have the same singular numbers,
we similarly obtain that $s_{D(\eps_1,\eps_2)}$ is decreasing
in $\eps_2$.
Now letting $\eps_1,\eps_2\to0$ and using the monotone convergence theorem in~\eqref{eq:intlog<Deps} and~\eqref{eq:intlogDeps<},
we obtain the desired inequalities~\eqref{eq:intlog<D} and~\eqref{eq:intlogD<} for our $A$, $B$ and $D$.
\end{proof}

\section{The multiplicative Horn problem in finite Neumann algebras}
\label{sec:finvN}

In this section we solve the multiplicative Horn problem in finite von Neumann algebras.
The solution is analogous to the result in matrix algebras that was proved in Section~\ref{sec:non-inv}.

Let $\SVF$ denote the set of all
real-valued, right-continuous, non-negative, non-increasing functions on $[0,1]$.
These are the functions that can be singular value functions of elements in finite von Neumann algebras.

For a Horn triple $(I,J,K)\in H(n,r)$, we will make use of the subsets $F_I$ of $[0,1]$ introduced in~\eqref{eq:FI}
at the beginning of Section~\ref{sec:sing}.

Let $f,g\in\SVF$.
Let
$M_{f,g}$ be the set of all $h\in\SVF$ such that
there exists a diffuse, finite von Neumann algebra $\Mcal$ with normal, faithful tracial state $\tau$ and there exist $A,B\in\Mcal$ 
yielding singular number functions $s_A=f$, $s_B=g$  and $s_{AB}=h$.
Our goal is to describe the set $M_{f,g}$.

Mimicking the finite dimensional case,
we define
\begin{align*}
\Mt_{f,g}=\bigg\{h\in\SVF\;\bigg|\;&\forall n\in\Nats\;\forall(I,J,K)\in\bigcup_{r=1}^n H(n,r), \\
&\int_{F_I}\log f+\int_{F_J}\log g\leq \int_{F_\Kbar}\log h \\
&\text{and }
\int_{(F_I)^c}\log f+\int_{(F_J)^c}\log g\geq \int_{(F_\Kbar)^c}\log h
\bigg\},
\end{align*}
where of course, the definitions of $F_I$, {\em etcetera}, in the above inequalities depend on the value of $n$ under consideration,
where $-\infty$ is allowed for values of the integrals
and where the complements are taken in $[0,1]$.

Our main result is:
\begin{thm}\label{thm:MMt}
For all $f,g\in\SVF$, we have $M_{f,g}=\Mt_{f,g}$
\end{thm}
\begin{proof}
The inclusion $M_{f,g}\subseteq\Mt_{f,g}$ follows from Theorem \ref{thm:AB}.

To show the reverse inclusion.
Let $f,g\in\SVF$ and let $h\in\Mt_{f,g}$.
For a function $s\in\SVF$ and $n\in\Nats$, we let
$s^{(n)}\in\Reals_+^{n\downarrow}$ be the sequence whose $j$-th element is
\[
s^{(n)}_j=\exp\left(n\int_{(j-1)/n}^{j/n}\log s(x)dx\right).
\]
Now from $h\in\Mt_{f,g}$, we easily verify $h^{(n)}\in\Kt_{f^{(n)},g^{(n)}}$,
and using Theorem~\ref{thm:KKt}, we deduce that there are
matrices
$A_n,B_n\in M_n(\Cpx)$
with the property that the singular values of $A_n$ and, respectively, of $B_n$ and $A_nB_n$
are the sequences $f^{(n)}$ and, respectively, $g^{(n)}$ and $h^{(n)}$.
Now, letting $\omega$ be a free ultrifilter on $\Nats$ and letting $\Mcal=\prod_\omega M_n(\Cpx)$
be the corresponding ultraproduct of matrix algebras,
letting $A,B\in\Mcal$ be the elements represented by sequences $(A_n)_{n=1}^\infty$ and $(B_n)_{n=1}^\infty$, respectively,
we have that the singular value functions of $A$, $B$ and $AB$ are, respectively, $f$, $g$ and $h$.
Thus, $h\in M_{f,g}$.
\end{proof}

Let us conclude by expanding on the remark made in the penultimate paragraph of the introduction about Connes' embedding property. 
For $f,g\in\SVF$, let $M_{f,g}^{\text{emb}}$ be the set of all $h\in\SVF$
such that there exist $A$ and $B$ in an ultrapower $R^\omega$ of the hyperfinite II$_1$-factor,
with singular value functions $s_A=f$, $s_B=g$ and $s_{AB}=h$.
We clearly have $M_{f,g}^{\text{emb}}\subseteq M_{f,g}$, while
the proof of the above theorem actually showed $\Mt_{f,g}\subseteq M_{f,g}^{\text{emb}}$.
Thus, we get:
\begin{cor}
For any $f,g\in\SVF$, we have
$M_{f,g}=M_{f,g}^{\text{emb}}$.
\end{cor}



\begin{bibdiv}
\begin{biblist}

\bib{Be01}{article}{
  author={Belkale, Prakash},
  title={Local systems on $\mathbb{P}^1-S$ for $S$ a finite set},
  journal={Compos. Math.},
  volume={129},
  year={2001},
  pages={67-86}
}

\bib{BCDLT10}{article}{
  author={Bercovici, Hari},
  author={Collins, Beno\^it},
  author={Dykema, Ken},
  author={Li, Wing Suet},
  author={Timotin, Dan},
  title={Intersections of Schubert varieties and eigenvalue inequalitites in an arbitrary finite factor},
  journal={J. Funct. Anal.},
  year={2010},
  volume={258},
  pages={1579--1627}
}

\bib{BL01}{article}{
  author={Bercovici, Hari},
  author={Li, Wing-Suet},
  title={Inequalities for eigenvalues of sums in a von Neumann algebra},
  conference={
    title={Recent advances in operator theory and related topics},
    address={Szeged},
    date={1999},
  },
  book={
    series={Oper. Theory Adv. Appl.},
    volume={127},
    year={2001},
    publisher={Birkh\"auser},
    address={Basel},
  },
  pages={113--126}
}

\bib{BLT09}{article}{
   author={Bercovici, H.},
   author={Li, W. S.},
   author={Timotin, D.},
   title={The Horn conjecture for sums of compact selfadjoint operators},
   journal={Amer. J. Math.},
   volume={131},
   date={2009},
   pages={1543--1567},
}

\bib{CD09}{article}{
  author={Collins, Beno\^it},
  author={Dykema, Ken},
  title={On a reduction procedure for Horn inequalities in finite von Neumann algebras},
  journal={Oper. Matrices},
  volume={3},
  year={2009},
  pages={1--40},
}

\bib{FaKo86}{article}{
   author={Fack, Thierry},
   author={Kosaki, Hideki},
   title={Generalized $s$-numbers of $\tau$-measurable operators},
   journal={Pacific J. Math.},
   volume={123},
   date={1986},
   pages={269--300},
}
		
\bib{FuKa52}{article}{
  author={Fuglede, Bent},
  author={Kadison, Richard V.},
  title={Determinant theory in finite factors},
  journal={Ann. of Math.},
  volume={55},
  year={1952},
  pages={520--530}
}

\bib{Fulton00}{article}{
  author={Fulton, William},
  title={Eigenvalues, invariant factors, highest weights, and Schubert calculus},
  journal={Bull. Amer. Math. Soc. (N.S.)},
  volume={37 (3)},
  year={2000},
  pages={209--249}
}

\bib{H07}{article}{
   author={Harada, Tetsuo},
   title={Multiplicative versions of Klyachko's theorem in finite factors},
   journal={Linear Algebra Appl.},
   volume={425},
   date={2007},
   pages={102--108},
}

\bib{H62}{article}{
  author={Horn, Alfred},
  title={Eigenvalues of sums of Hermitian matrices},
  journal={Pacific J. Math},
  volume={12},
  year={1962},
  pages={225--241}
}

\bib{Kl98}{article}{
  author={Klyachko, Alexander A.},
  title={Stable bundles, representation theory and Hermitian operators},
  journal={Selecta Math. (N.S.)},
  volume={4},
  year={1998},
  pages={419--445}
}

\bib{Kl00}{article}{
  author={Klyachko, Alexander A.},
  title={Random walks on symmetric spaces and inequalities for matrix spectra},
  journal={Linear Algebra Appl.},
  volume={319},
  year={2000},
  pages={37--59}
}

\bib{KT99}{article}{
  author={Knutson, Allen},
  author={Tao, Terrence},
  title={The honeycomb model of GL$_{n}(\mathbb{C})$ tensor products. I. Proof of the saturation conjecture},
  journal={J. Amer. Math. Soc.},
  volume={12},
  year={1999},
  pages={1055--1090}
}

\bib{Sp05}{article}{
   author={Speyer, David E.},
   title={Horn's problem, Vinnikov curves, and the hive cone},
   journal={Duke Math. J.},
   volume={127},
   date={2005},
   pages={395--427},
}

\bib{W12}{article}{
  author={Weyl, H.},
  title={Das asymptotische Verteilungsgesetz der Eigenwerte lineare partieller Differentialgleichungen},
  journal={Math. Ann.},
  volume={71},
  year={1912},
  pages={441--479},
}

\end{biblist}
\end{bibdiv}
\end{document}